\documentclass[11pt]{amsart}
\usepackage{amsthm}
\usepackage{amssymb}
\usepackage{amsmath}
\usepackage{amscd}
\usepackage{epic}
\usepackage{eepic}
\usepackage{xy} 
\usepackage{hyperref}

\newcommand {\comment}[1]{}

\input xy 
\xyoption{all}

\def    \N  {{\mathbb N}}

\def    \Z  {{\mathbb Z}}

\def    \R  {{\mathbb R}}

\def    \C  {{\mathbb C}}

\def    \CP {{\mathbb C}{\mathbb P}}
\def    \tCP {{\mathbb C}{\mathbb P}{^2}}

\def    \oCP {{\mathbb C}{\mathbb P}{^1}}

\def    \reg   {{reg}}
\def    \gen   {{sp}}

\def    \half   {{\frac{1}{2}}}

\def    \s {{\mathcal{S}}}
\def    \M {{\mathcal{M}}}
\def    \z {{\mathcal{Z}}}
\def    \F  {{\mathcal{F}}}
\def    \H  {{\mathcal{H}}}
\def    \J  {{\mathcal{J}}}
\def    \B  {{\mathcal{B}}}
\def    \E  {{\mathcal{E}}}

\def   \eoe  {{\unskip\ \hglue0mm\hfill$\between$\smallskip\goodbreak}}

\DeclareMathOperator {\Id} {Id}
\DeclareMathOperator {\id} {Id}

\DeclareMathOperator {\pr} {pr}
\DeclareMathOperator {\cl} {cl}

\DeclareMathOperator {\FS} {FS}
\DeclareMathOperator {\Striv} {Striv}
\DeclareMathOperator {\pt} {pt}

\swapnumbers

\newtheorem{Lemma}{Lemma}
\newtheorem{Theorem}[Lemma]{Theorem}
\newtheorem{Proposition}[Lemma]{Proposition}
\newtheorem{Corollary}[Lemma]{Corollary}
\newtheorem{Claim}[Lemma]{Claim}
\newtheorem {Definition}[Lemma]{Definition}
\theoremstyle{remark} 
\newtheorem{Remark}[Lemma]{Remark}
\newtheorem{Example}[Lemma]{Example}
\newtheorem{noTitle}[Lemma]{}

\numberwithin{Lemma}{section}

\begin{document}
\title[Holomorphic shadows in the eyes of model theory]{Holomorphic shadows in the eyes of model theory}
\author{Liat Kessler}
\address{M.I.T. Department of Mathematics, Cambridge MA 02139, U.S.A}
\email{kessler@math.mit.edu}

\begin{abstract}

We define a subset of an almost complex manifold $(M,J)$ to be a \emph{holomorphic shadow} if it is the image of
a J-holomorphic map from a compact complex manifold. 
Notice that a J-holomorphic curve is a holomorphic shadow, and so is a complex subvariety of a compact complex manifold.

We show that under some conditions on an almost complex structure $J$ on a manifold $M$,
the holomorphic shadows in the Cartesian products of $(M,J)$  
form a Zariski-type structure. 
Checking this leads to non-trivial geometric questions and results. 
We then apply the work of Hrushovski and Zilber on Zariski-type structures. 

We also restate results of Gromov and McDuff on $J$-holomorphic curves in symplectic geometry in the language of shadows structures. 
\end{abstract}

\maketitle

\section{Introduction}

An almost complex manifold $(M,J)$ is a manifold $M$ with a complex structure $J$
on the fibers of the tangent bundle $TM$. 
A smooth ($C^{\infty}$) map is J-holomorphic if for every point in the domain the differential is a 
complex linear map between the tangent spaces. In this paper we study almost complex manifolds and J-holomorphic  maps using the framework of Zariski-type structures form model theory.

Zariski-type structures were
introduced and studied by Zilber and Hrushovski \cite{imam, imam1}, \cite{zilber}, \cite{ezilber} in their study of strongly minimal sets.
A \emph{Zariski-type structure} is a set $M$ with a collection 
of compatible Noetherian 
topologies, one on each $M^{n}$ for $n \in \N$, 
with an assignment of  \emph{dimension} to the closed sets, 
satisfying certain conditions that are reasonable to require if we think of the closed sets as subvarieties.
A topology is called \emph{Noetherian} if the Descending Chain Condition holds for closed sets; by \emph{compatible} we mean that the coordinate projections are continuous and closed.
Zilber showed that such a structure admits elimination of 
quantifiers, which is essential in applications of abstract model theory in concrete areas of mathematics.
 A motivating example is given by taking the complex subvarieties of Cartesian products of a 
compact complex manifold to be the closed sets and the dimension to be the complex dimension. 

One motivation to our study is to give a good definition for an almost complex subvariety; namely, a holomorphic shadow. The interpretation of ``good' is according to the axioms of a Zariski-type structure. Then we can apply results from model theory to characterize an almost complex manifold by properties of the structure of holomorphic shadows in its Cartesian products.

In Section \ref{defalmost} we give the necessary background on almost complex manifolds and J-holomorphic maps.
 We denote   by $\J_{\gen}$ the set of almost complex structures $J$ on $M$ such that  there are no $J$-holomorphic maps from complex manifolds to $(M,J)$ except for curves and constant maps  (see \eqref{jgen}). 
We define a holomorphic shadow and the holomorphic shadows structure. In Section \ref{zariski} we present the axioms of a Zariski-type structure, as defined by \cite{imam}, \cite{zbook}, and prove that for $J \in {\J}_{\gen}$ the holomorphic shadows in Cartesian products of $(M,J)$ form a Zariski-type structure.
At the end of Section \ref{zariski} we give an immediate application of the work of Hrushovski and Zilber to our case.

Another goal of our study is to restate results  from Gromov's theory of J-holomorphic curves in symplectic manifolds in the language of shadows structures, as we do in
Section \ref{sympl}.


\section{Almost complex manifolds and holomorphic shadows} \label{defalmost}

\subsection*{Almost complex manifolds and maps}
An \emph{almost complex structure} on a $2n$-dimensional manifold $X$ is an
automorphism of the tangent bundle, $J \colon TX \to TX$,
such that $J^2 = -\id$. The pair
$(X,J)$ is called an \emph{almost complex manifold}. 
An almost complex structure is \emph{integrable} if it is induced from a complex 
manifold structure.
In dimension two any almost complex manifold is integrable (see, e.g., \cite[Theorem 4.16]{MS:intro}). 
In higher dimensions this is not true \cite{calabi}. 
A submanifold $Y$ of $X$ is called an \emph{almost complex submanifold} if $J TY = TY$. 
We denote by $\J(X)$ the space of all almost complex structures on $X$ with the $C^{\infty}$ topology.

A smooth ($C^{\infty}$) map $f \colon X_{1}\rightarrow X_{2}$ is called \emph{J-holomorphic} if for all $p \in X_{1}$ the differential $df_p \colon T_{p}(X_{1})\rightarrow T_{f(p)}(X_{2})$ is a 
complex linear map, i.e.,
\[df_{p}\circ {J_{1}}_{p} = {J_{2}}_{f(p)}\circ df_{p}.\]
This coincides with the Cauchy Riemann equations if $(X_{1},J_{1})$ and $(X_{2},J_{2})$ are complex manifolds.  
The equation for holomorphic maps between two almost-complex manifolds becomes overdetermined as soon as the complex dimension of
the domain exceeds one, 
so for a generic almost complex structure $J$ on a manifold $X$, there should not be any almost complex submanifolds of complex dimension strictly greater than one. 
We denote by 
\begin{equation} \label{jgen}
\J_{\gen}
\end{equation}
 the subset of almost complex structures such that for every $J$-holomorphic map from a holomorphic disc $\C^k$ to $(X,J)$ there is a neighbourhood of $0$ in $\C^k$ such that the map factors through $\C^k \to \C$.

When the domain of a J-holomorphic map is a compact Riemann surface (i.e., a compact one-dimensional complex manifold), we call the map a \emph{parameterized J-holomorphic curve}, its image is called  a \emph{J-holomorphic curve}.
 When the domain is $\CP^1$, with the standard complex structure, the map is a 
 \emph{parametrized J-holomorphic sphere}. A J-holomorphic map is called \emph{simple} if it cannot be factored through a branched 
covering of the domain. 
In general, a J-holomorphic curve cannot be represented 
as the common zeroes of J-holomorphic functions into $\C$, not even locally. This makes the notion of an 
almost complex variety tricky.

\begin{noTitle} \label{du}
Let $(X,J)$ be an almost complex manifold.
Let  $$\Sigma_k=(\Sigma_k,j)$$ be a compact complex manifold of complex dimension $k$. The J-holomorphic maps from $(\Sigma_k,j)$ to $(X,J)$ are the maps satisfying 
$$\bar{\partial}_{J}(u)=0,$$ where $$\bar{\partial}_J(u) := \half (du + J \circ du \circ j).$$
Let $A \in H_{2k}(X;\Z)$ be a homology class.
The $\bar{\partial}_J$ operator defines a section
$S \colon \B \to \E$ by
$$ S(u):=(u,\bar{\partial}_J(u)) $$
where $\B \subset C^{\infty}(\Sigma_k,X)$ denotes the space of all smooth maps $u \colon \Sigma_k \to X$ that represent the homology class $A$, and the bundle $\E \to \B$ is the infinite dimensional vector bundle whose fiber at $u$ is the space $\E_u = \Omega^{0,1}(\Sigma_k,u^{*}TX)$ of smooth $J$-antilinear $1$-forms on $\Sigma_k$ with values in $u^{*}TX$. 
The moduli space 
$$ \M(A, \Sigma_k ,J)=\{ u  \ \mid \  u \text{ is a }(j,J)\text{-holomorphic map }\Sigma_k \to X \text{ in }A \} $$ is the zero set of this section. 
Denote by
\begin{equation} \label{odu}
D_u=DS(u) \colon \Omega^{0}(\Sigma_k,u^{*}TX) \to \Omega^{0,1}(\Sigma_k,u^{*}TX)
\end{equation}
the composition of the differential $dS(u) \colon T_{u}{\B} \to T_{(u,0)}{\E}$ with the projection $\pi_u \colon T_{(u,0)}=T_{u}{\B} \oplus {\E}_u \to {\E}_u$.
The operator $D_u$ is the \emph{vertical differential} of the section $S$ at $u$.

If $k=1$, then $D_u$ is a real linear Cauchy Riemann operator.
When $k >1$, the image of the map \eqref{odu} is of infinite codimension.

Consider the universal moduli space
$$ \M(A,\Sigma_k,\J)=\{(u,J) \ \mid \ J \in \J, \ u \text{ is a }(j,J)\text{-holomorphic map }\Sigma_k \to X \text{ in }A \} .$$
Here $\J$ is an open subset of $\J(X)$. When $X$ has a symplectic form $\omega$ (see Section \ref{sympl}), we can take $\J$ to be the space of all 
$\omega$-tame almost complex structures.  

Consider the projection map
$$p_{A} \colon \M(A,\Sigma_k,\J) \to \J.$$
The differential $dp_A$ at a point $(u,J)$ is essentially the operator $D_u$, and is surjective at $(u,J)$ when $D_u$ is onto. 

We have the following consequence of the 
Sard-Smale theorem,
the infinite dimensional implicit function theorem,
and the ellipticity of the Cauchy-Riemann equations.
When $k=1$, 
the set $\J_{\reg}(A)$ of regular values for $p_A$ is of the second category in $\J$; for any $J \in \J_{\reg}(A)$, the space of simple 
$J$-holomorphic $\Sigma_1$-curves in $A$ is a smooth manifold of dimension $2c_1(A)+n(2-2g)$, where $c_1$ is the first Chern class of the complex vector bundle $(TX,J)$, 
$2n$ is the dimension of $X$, and $g$ is the genus of $\Sigma_1$
 \cite[Theorem 3.1.5]{MS:JCurves}. 
When $k > 1$, for a generic $J$ the space 
$\M(A,\Sigma_k,J)$ is empty.
\end{noTitle}

\subsection*{Holomorphic shadows}
  
      \begin{Definition}
     A subset of an almost complex manifold $(X,J)$ is a \emph{holomorphic shadow} \footnote{We follow an attempt of Hardt \cite{ha2} to give the name \emph{semianalytic shadows} 
to subanalytic sets} 
    if it is the image of a J-holomorphic map from a compact complex analytic manifold.
    \end{Definition}
    
\begin{Remark}    
As a compact set in a Hausdorff space, each holomorphic shadow is closed in the $C^{\infty}$ topology on $X$.
\end{Remark}

\begin{noTitle} \label{rest}
For a complex analytic subvariety $V$ of a complex analytic manifold $(M,J_M)$, i.e., a subset given locally as the common zeros of a 
finite collection of holomorphic functions,
we say that 
a map $f \colon V \rightarrow X$ is J-holomorphic if for one (hence every) resolution 
of $V$ to a complex analytic manifold $\tilde{V}$, 
${\phi} \colon \tilde{V} \rightarrow V$, the map
$f \circ {\phi}$ is a proper J-holomorphic map from the complex analytic 
manifold $\tilde{V}$. 
\end{noTitle}

By \cite{hiro}, every complex analytic subvariety admits a resolution of singularities, i.e.,  a map ${\phi} \colon \tilde{V} \rightarrow V$, such that $\tilde{V}$ is a complex analytic manifold, the preimage of the nonsingular points of $V$ is a dense subset in $\tilde{V}$ on which $\phi$ is an isomorphism, 
and $\phi$ is a proper map, (in particular, if $V$ is compact so is $\tilde{V}$). 
On the other hand, by the proper mapping theorem, 
 an image of a 
complex analytic subvariety 
by a proper holomorphic map is a complex analytic subvariety. As a result we get the following claim.

\begin{Claim} \label{subim} 
A subset of a compact complex analytic manifold is a complex analytic subvariety if and only if it is a  holomorphic shadow.
\end{Claim}

For a compact complex analytic manifold $M$, taking the complex analytic subvarieties of $M^n$, $n \in \N$, to be the closed subsets and the dimension to be the complex dimension gives a Zariski-type structure. This follows from standard facts in complex geometry, as observed by B.\ Zilber \cite{imam1}. 
We show a similar claim in the non-integrable case.

\begin{Definition}\label{str}
Given an almost complex $2{r}$-manifold $(X,J)$ and a collection $\H$ of holomorphic shadows in the finite Cartesian products of $(X,J)$,
we consider the collection of:
 \begin{itemize}
  \item the holomorphic shadows in $\H$,
  \item the diagonals, i.e., sets of the form
     $${{\Delta}^{n}}_{(i_{1},\ldots,i_{k})}= \{\bar{x} \in X^{n} \mid x_{i_{1}}= \ldots = x_{i_{k}}\},$$
  \item subsets of $X^n$ of the form $S \times D_1 \times \ldots  \times D_k$, where $S \in \H$ is a
        shadow in $X^l$, each $D_i$ is a diagonal in $X^{d_i}$, and $\sum_{i=1}^{k}{d_i}=n-l$.
  \item the images of sets as above under permutations of the coordinates,
  \item finite unions of the above sets.
\end{itemize}
We denote this collection ${\s}_{(X,J,\H)}$. 
 \end{Definition}

{\bf Notation}
when $\H$ is the collection of all holomorphic shadows in the finite Cartesian products of $(X,J)$, we write ${\s}_{(X,J)}$ for
${\s}_{(X,J,\H)}$. We call  ${\s}_{(X,J)}$ the \emph{holomorphic shadows structure}.

The holomorphic shadows  structure admits a natural (partial) dimension function:
\begin{itemize}
  \item the dimension of a point is $0$;
  \item the dimension of a non-constant J-holomorphic curve is $1$;
 \end{itemize}

\begin{Theorem} \label{genz}
Let $(X,J)$ be an   almost complex manifold. 
Assume that $J \in \J_{\gen}$. Then there exists a dimension function $\dim$ on ${\s}_{(X,J)}$ that is consistent with the natural partial
 dimension above, such that 
$(X,{\s}_{(X,J)},\dim)$ is a Zariski-type structure that satisfies 
the essential uncountability (EU) property. 
\end{Theorem}

  The set $\J_{\gen}$ is defined in \eqref{jgen}.
  We will give explicit definitions and prove the theorem in the next section.

\section{Zariski-type structures} \label{zariski}
A Zariski-type structure, or a Z-structure, as defined in \cite{imam}, is a set $X$ with a collection $\mathcal{C}$ of subsets of its Cartesian products, 
$X^n$, to be called 
\emph{Z-closed} sets, and a dimension assignment to $\mathcal(C)$, such that:
                \begin{itemize}
         \item (L1) The set $X$ is Z-closed;
         \item (L2) Each point is Z-closed;
         \item (L3) Cartesian products of Z-closed sets are Z-closed; 
         \item (L4) The diagonals are Z-closed;
         \item (L5) Finite unions and intersections of Z-closed sets are Z-closed;
         \item (P) Any of the coordinate projections 
                $$\pr_{i_1,\ldots,i_m} \colon (x_1,\ldots,x_n) \mapsto (x_{1_1},\ldots,x_{i_m}), \, i_1,\ldots, i_m \in \{1, \ldots,n \} $$ 
                   are
          closed and continuous, i.e., the images and inverse images
          of Z-closed sets under these projections are Z-closed;
                \item (DCC) Descending Chain Condition for Z-closed sets: 
          For any Z-closed
          \[C_{1} \supseteq C_{2} \supseteq \ldots \supseteq C_{i} \supseteq \ldots \]
          there is $k$ such that $C_{i} = C_{k}$ for all $i \geq k$. 
                \end{itemize}
          This condition implies that for any Z-closed $C$ there are
                                Z-closed $C_{1},\ldots,C_{m}$, that are distinct and  
                no one is a subset of the other, 
such that 
                                $C=C_{1} \cup \ldots \cup C_{m}$, where $m$ is maximal.
                     These $C_{i}$ are the \emph{irreducible components} of
                     $C$. They are defined up to permutation uniquely.
                   
A Z-closed set $S$ is called \emph{irreducible} if there are no Z-closed subsets $S_1,S_2 \subsetneq S$ such that $S=S_1 \cup S_2$.

 To any Z-closed subset $C$, there is attached a natural number, called $\dim C$,
                     such that:
         \begin{itemize}
                \item (DP) Dimension of a Point is $0$;
                \item (DU) Dimension of a Union: $\dim(C_{1} \cup C_{2}) = \max\{\dim C_{1},\dim C_{2}\};$
                \item (DI) $\dim C_{1} < \dim C$ for $C$ irreducible,
                           $C_{1} \subseteq C, C_{1} \neq C;$
   \item (FC) For any $k \in \N$, for any Z-closed $C \subset X^n$, and projection $pr \colon X^n \to X^m$, the set  
$$p(C,k)=\{a \mid \dim(C \cap {pr}^{-1}(a)) > k \}$$ is constructible.
 \item (ADF) For any irreducible Z-closed $C \subset X^n$ and projection $pr \colon X^n \to X^m$, 
    $$\dim C = \dim pr(C) + \min_{a\in pr(C)}\dim (pr^{-1}(a) \cap C).$$ 
    \end{itemize}
A \emph{constructible} set is a  
finite Boolean combination of Z-closed sets.

We will call these axioms the \emph{Z axioms}.

Other properties that will be relevant are the following:
        \begin{itemize}
         \item (EU) Essential uncountability: If a Z-closed $C \subseteq X^{n}$ is a union of countably
                 many Z-closed subsets, then there are finitely many among the subsets whose union is $C$.
      This implies that if $X$ is not finite it must be uncountable.
                \item (PS) Pre-smoothness: For any irreducible Z-closed $S_{1},S_{2} \subseteq X^{n}$,
                the dimension of any irreducible component of
     $S_{1} \cap S_{2}$ is no less than $\dim(S_{1}) + \dim (S_{2}) - \dim X^{n}.$
         \end{itemize}

In  the Zariski-type structure in which the Z-closed sets are the complex subvarieties of Cartesian products of a 
compact complex manifold, and the dimension is the complex analytic dimension, the properties (EU) and (PS) are satisfied \cite{imam}, \cite{zbook}.

 Zilber \cite{imam1}, \cite{zbook} showed that any Zariski-type structure admits elimination of quantifiers: the projection of a constructible set
is constructible.
     
\subsection*{Proof of Theorem \ref{genz}}

It follows from the definition of a J-holomorphic map that
\begin{Claim}
\begin{itemize}
\item J-holomorphic maps are closed under disjoint union,  Cartesian product and composition (when defined).
\item The canonical coordinate projections $ \pi \colon X^{n+k} \to X^n$ are J-holomorphic.
\end{itemize}
\end{Claim} 

As a corollary we get the following claim.

\begin{Claim} \label{123}
\begin{enumerate}
\item A finite union of holomorphic shadows is a holomorphic shadow.\\
\item A finite Cartesian product of  holomorphic shadows is a holomorphic shadow. \\
\item The image of a  holomorphic shadow under a J-holomorphic map 
is a holomorphic shadow.\\
\item The image of a holomorphic shadow under the canonical coordinate projection $X^{n+k} \to X^n$ 
       is a holomorphic shadow.
\end{enumerate}
\end{Claim}

To continue, we show the following Lemma.

\begin{Lemma} \label{dgen}
Let $(X,J)$ be an  almost complex manifold. Assume 
that $J \in \J_{\gen}$. Then any J-holomorphic map $f \colon M_f \to X^{n(f)}$ from a compact complex manifold to a Cartesian product of $(X,J)$ satisfies the ``pulling back diagonals property'': the preimage of any diagonal $\Delta^{n(f)}_{i_1,\ldots,i_k}$ is a complex subvariety of $M_f$.
\end{Lemma}

\begin{proof}
We first show that the pulling back diagonals property holds for J-holomorphic maps of the form
\begin{equation} \label{form}
\prod_{j=1}^{n}{g_j} \colon \prod_{j=1}^{n}\Sigma^{(j)} \to X^n,
\end{equation}
 where the 
$\Sigma^{(j)}$-s are compact connected Riemann surfaces and the maps $g_j \colon \Sigma^{(j)} \to X$  are J-holomorphic and satisfy the following assumptions:
\begin{equation}  \label{assumption1}
g_j \text{ is simple } \forall j,
\end{equation}
and 
\begin{equation}  \label{assumption2}
g_{i_1}(\Sigma^{i_1})=g_{i_2}(\Sigma^{i_2}) \Rightarrow  \Sigma^{(i_1)}=\Sigma^{(i_2)} \text{ and }g_{i_1}=g_{i_2}.
\end{equation}

Indeed, for an almost complex manifold $(X,J)$, for two $J$-holomorphic maps
$g_i  \colon \Sigma^{(i)} \to X$, $i=1,2$ from compact 
connected 
Riemann surfaces $\Sigma^{(i)}$, either 
the set $$(g_1,g_2)^{-1}(\Delta_{1,2}^2(X))=\{(z_1,z_2) \in \Sigma^{(1)} \times \Sigma^{(2)} \ \mid \ g_1(z_1)=g_2(z_2)\}$$ consists of finitely many points, or
$$g_1(\Sigma^{(1)})=g_2(\Sigma^{(2)}), \text{ hence, by \eqref{assumption2} } \Sigma^{(i_1)}=\Sigma^{(i_2)} = \Sigma \text{ and }g_{i_1}=g_{i_2}.$$
For a simple J-holomorphic $g \colon \Sigma \to X$ from a compact Riemann surface, 
$$(g,g)^{-1}(\Delta_{1,2}^{2}(X))= \Delta_{1,2}^{2}(\Sigma) \cup S,$$ where $S$ consists of finitely many points. 
See \cite[Theorem E.1.2, Exercise E.1.4]{MS:JCurves} and \cite{malta}.
By induction, this implies that the pulling 
back diagonals property holds for maps of the form \eqref{form}.

Now, for $J \in \J_{\gen}$, a holomorphic shadow in $(X,J)$ is either a J-holomorphic curve or a point.
Hence for a holomorphic shadow $S \subset (X^n,J^n)$, the projection of $S$ on any of the coordinates is either a point or a $J$-holomorphic curve. 
So every $J$-holomorphic map $f \colon M \to X^n$ from a compact complex manifold $M$ decomposes as 

\xymatrix{
 M  \ar[d]^{h}  \ar[r]^{f} & f(M)   \ar[d]^{\prod_{j=1}^{n}{\pr_{j}|_{f(M)}}}  \ar[r]^{\subset} &X^n \\
 {\prod_{j=1}^{n} \Sigma^{(j)}}  \ar[r]_{\prod_{j=1}^n{g_j}} & {\prod_{j=1}^{n} C_j } \ar[r]_{\subset} & X^n  }

where $h \colon M \to  {\prod_{j=1}^n \Sigma^j}$ is a holomorphic map, and for all $j$, $g_j \colon \Sigma^{(j)} \to X$ is a $J$-holomorphic map from a compact Riemann surface $\Sigma^{(j)}$. 
Some of these maps might be constant, in that case replace $\Sigma^{(j)}$ with a point.
We can also assume that \eqref{assumption1} and \eqref{assumption2} hold. This reduces this case to the special case \eqref{form} discussed above.

\end{proof}

We show that the pulling back diagonals property imply 
being closed under finite intersections, the Descending Chain Condition, the fact that the image of an irreducible set under a coordinate projection is irreducible, and the Essential Uncountability  
in ${\s}_{(X,J)}$.

{\bf Notation:}
In this subsection, $\F$ is the collection of all J-holomorphic maps from compact complex  manifolds to Cartesian products of $(X,J)$ for  $J \in {\J}_{\gen}$, and $\H$ denotes the collection of all holomorphic shadows in the finite Cartesian products of $(X,J)$ for $J \in \J_{\gen}$.

\begin{Claim} \label{inc} 
 Consider
  \[    f_{1} \colon  M_{1} \rightarrow X^n   \]
\[    f_{2} \colon  M_{2} \rightarrow X^n,  \]
maps in $\F$.
Then ${f_1}^{-1}[f_1(M_1) \cap f_2(M_2)] $ is a complex subvariety of $M_1$. 
 \end{Claim} 

\begin{proof}
By the ''pulling back diagonals`` property, 
$Z=(f_{1} \times f_{2})^{-1}[\Delta_{X^n}]$ is a complex subvariety
of the complex manifold $M_{1} \times M_{2}$.

The preimage ${f_{1}}^{-1}[f_{1}(M_1) \cap f_2(M_2)]$
is the image of the complex subvariety $Z$ by the  canonical projection
$\pi_{1} \colon M_{1} \times M_{2} \rightarrow M_{1}$, hence, by the proper mapping theorem,
it is a complex subvariety of $M_1$.

\end{proof}
Now, the intersection $f_1(M_1) \cap f_2(M_2)$ is the image of 
$f_1 \circ \pi_1 \circ \phi$, where $\phi \colon \tilde{Z} \to Z$ is a resolution of 
$Z=(f_{1} \times f_{2})^{-1}[\Delta_{X^n}]$ to a compact complex analytic manifold $\tilde{Z}$.

As a result, 
\begin{Corollary}
The holomorphic shadows in $\H$ are closed under finite intersections.
\end{Corollary}

Similarly, consider
        \[    f_{1} \colon  M_{1} \rightarrow X^{n+k} \]
        \[    f_{2} \colon  M_{2} \rightarrow X^n,    \]
         in $\F$, and the coordinate projection map 
         \[ \pr_{1,\ldots,n} \colon X^{n+k} \to X^n. \]

     The preimage
     $Z=(\pr_{1,\ldots,n} \circ f_{1},f_{2})^{-1}[\Delta_{X^n}]$ is a complex subvariety
     of the complex manifold $M_{1} \times M_{2}$, hence its image by $f_1$ composed on  
     the  canonical projection ${{M_{1} \times M_{2} \rightarrow M_{1}}}$ (composed on a resolution of $Z$) is a holomorphic shadow $\in \H$
     in $X^{n+K}$. 
     
     \begin{Corollary} \label{int2}
     For $S_1 \in \H $ a holomorphic shadow in $X^{n+k}$, and $S_2 \in \H$ a holomorphic shadow in $X^n$, 
     the intersection $S_1 \cap S_2 \times X^k$ is a holomorphic shadow $\in \H$ in $X^{n+k}$.
     \end{Corollary}
     \comment{generalize to product with a diagonal}
    
     \begin{Corollary} \label{intj}
    Let $A \in \H$ be a holomorphic shadow in $X^n$. Let $D$ be a diagonal in $X^n$. Then the intersection $A \cap D$
    is a holomorphic shadow in $X^n$.
\end{Corollary}

\begin{proof}
$A$ is the image under a J-holomorphic 
   map from a compact complex manifold $M$ into $X^n$. Then, $A \cap D$
   is the image under $f$ of $f^{-1}[D]$. By Lemma
       \ref{dgen}, $f^{-1}[D]$ is a complex subvariety in $M$.
       Hence, this image is a
       holomorphic shadow.
\end{proof}

\begin{Claim} \label{inv}
Let $S \in \H$ be a holomorphic shadow in $X^{n+m}$. 
Then the inverse image of a holomorphic shadow $C \subseteq X^n$ under the projection 
${\pr_{1,\ldots,n}}|_{S} \colon S \rightarrow X^n$ 
is a holomorphic shadow in $\H$. 
\end{Claim}

\begin{proof}
Given 
$$
 \begin{CD}
{M_S}    @> f_S  >>    {S} @> {\pr_{1,\ldots,n}}|_S >> {X^n}, \\ 
  \end{CD} $$
 set
$$A={\pr_{1,\ldots,n}}(S) \cap C \subseteq X^n.$$ By Claim \ref{inc}, 
$A$ is the image under $f_S$ of a complex subvariety 
 $M_A \subseteq M_S$.
It remains to notice that $f_S(M_A)={{\pr_{1,\ldots,n}}|_S}^{-1}(C)$.
\end{proof}

\begin{Claim} \label{DCC}
 The descending chain condition holds               
for holomorphic shadows in $\H$. 
\end{Claim}
\begin{proof}
Consider a descending chain 
\[ S_{1} \supseteq S_{2} \supseteq \ldots \supseteq S_{i} \supseteq \ldots \]
of holomorphic shadows $S_i=f_i(M_i)$ in $\H$. 
By Claim \ref{inc}, 
\[M_1= {f_1}^{-1}[S_{1}] \supseteq {f_1}^{-1}[S_1 \cap S_{2}] \supseteq \ldots \supseteq {f_1}^{-1}[S_1 \cap S_{i}] \supseteq \ldots \]
is a descending chain of complex subvarieties of $M_1$. 
Since $S_1 \cap S_i = S_i$, we get 
\[ {f_1}^{-1}[S_{1}] \supseteq {f_1}^{-1}[S_{2}] \supseteq \ldots \supseteq {f_1}^{-1}[S_{i}] \supseteq \ldots \]

By the descending chain condition                 
for complex subvarieties of a compact complex                 
manifold (see, e.g., \cite{imam1}), 
there is $k$ such that for all $i \geq k$ ${f_1}^{-1}[S_{i}] = {f_1}^{-1}[S_{k}],$ hence 
so are their images under $f_1$, i.e., 
$S_i=S_k$.

\end{proof}

This implies that for any holomorphic shadow $S$ in $\H$ there are           
distinct holomorphic shadows $S_{1},\ldots,S_{m}$ in $\H$ such that $S=S_{1}           
\cup \ldots \cup S_{m}$, where $m$ is maximal.           
These $S_{i}$ are the \emph{irreducible components} of $S$.

We make the following observation. 
\begin{Claim} \label{prirred} 
The image of an irreducible holomorphic shadow $C \in \H$ under a projection $pr \colon X^{n+m} \rightarrow X^n$ is an irreducible
 holomorphic shadow. \end{Claim}  
\begin{proof}

Otherwise ${pr}(C) = S_1 \cup S_2$ where $S_1$ and $S_2$ are distinct holomorphic shadows that $\neq {pr}(C)$.
By Claim \ref{inv}, ${{pr}|_C}^{-1}(S_1)$ and ${{pr}|_C}^{-1}(S_2)$ are holomorphic shadows in $\H$, they are distinct, $\neq C$ 
and their union equals $C$, to get a contradiction.
\end{proof}

\begin{Claim} \label{eu}
(EU) If a holomorphic shadow $S \in \H$ is a union of countably
                 many holomorphic shadows in $\H$, then there are finitely many among the
    subsets whose union is $S$.
\end{Claim}

 \begin{proof}
 Given $f \in \F$ such that
 $$ f(M_f) = S = \cup_{i \in \N}{S_i}, $$
 where $S_i$ are holomorphic shadows in $\H$, then $$ S \cap S_i = S_i, $$ so
 \[ M_f=f^{-1}[S]= f^{-1}[\cup_{i \in \N}{S_i}]=\cup_{i \in \N}f^{-1}[S_i] = \cup_{i \in  \N}f^{-1}[S \cap S_i]. \]
 By Claim \ref{inc}, 
for all $i$, the set $f^{-1}[S \cap S_i]$ is a complex subvariety of $M$.

 By the (EU) claim for complex subvarieties in a compact 
complex manifold (see \cite{imam1}),  there are finitely many among the
    subsets $f^{-1}[S \cap S_i]$ whose union is $M$, hence  
  there are finitely many among the
    subsets $S_i $ whose union is $S$. 
 \end{proof}

To complete the proof of Theorem \ref{genz}, we need to define a dimension of a holomorphic shadow, and show that it satisfies the dimension axioms. For that we need the following Lemma.

\begin{Lemma} \label{greq}
Let $A$ be a holomorphic shadow in $\H$. 
Then there is a subset $U_A$ in $A$ that satisfy the following.
\begin {enumerate}
\item $U_A$ is isomorphic to a complex manifold in the almost complex sense.
In particular, it is an integrable submanifold.
\item $U_A$ is dense in $A$, in the usual ($C^{\infty}$) topology.
\item The set $A \smallsetminus U_A$ is contained in a holomorphic shadow $P_A$ in $\H$ that is a proper subset of $A$.
\end{enumerate}
\end{Lemma}

\begin{proof}
By definition, $A$ is the image of a compact complex manifold $M$ under a J-holomorphic 
map 
\[f \colon M \rightarrow X^n.\]
By the ``pulling back diagonals property'', the inverse image under $(f,f) \colon M \times M \rightarrow X^n \times X^n$ of the diagonal of 
$X^n \times X^n$ is a complex subvariety in $M \times M$, i.e.,   
the relation ${\sim}_f$, where \[m_{1} \sim_f m_{2} \Leftrightarrow f(m_{1})=f(m_{2}),\] is a complex equivalence relation
in $M$. Moosa \cite{moosa2} showed that in this case there exist a degenerate 
complex subvariety $P$ of 
$M$ (i.e., the intersection of $P$ with each of the irreducible components of $M$
is a proper subvariety of the component, hence of a lower dimension),
 a compact complex analytic space $N$ with a degenerate complex subvariety $Q$ of N, and a holomorphic map 
\[g \colon U=M \smallsetminus P \rightarrow N, \] such that
 $N - g(U) \subset Q$ and for all $a,b \in U$, $g(a)=g(b)$ if and only if $a \sim_f b$.
The set $V=N \smallsetminus Q$ is a Zariski-open set that is dense in $N$ and contained in $g(U)$.  
By replacing $U$ with $g^{-1}V=M \smallsetminus P \smallsetminus g^{-1}Q$ we can assume that $g(U)=V$. By reducing $U$ and $V$, 
we can assume that $g$
 is a submersion, 
i.e., for any $u \in U$, $dg_u$ is onto $T_{g(u)}V$. By the local submersion theorem, 
for any $u \in U$, there are 
holomorphic  local coordinates around $u$ and $g(u)$ such that $g(u_1,\ldots,u_k)=(u_1,\ldots,u_l)$.

We define \[h \colon V \rightarrow X^n\] by $h(c) = x$ if there exists 
$m \in g^{-1}(c)$ such that $f(m) = x$. 
In the holomorphic local coordinates of $U$ and $V=g(U)$ chosen above,  
$h(u_1,\ldots,u_l)=f(u_1,\ldots,u_k)$, i.e., 
locally $f$ is the composition of $h$ with the
canonical submersion.  
Then, $h$ is a well defined one-to-one map 
that is smooth. 
Since  $$ J_{X^n}  dh \circ dg = J_{X^n}  d(h \circ g) = J_{X^n}  df = df J_{M} = d(h \circ g)  J_{X^n} = dh \circ dg   J_M = dh  J_N \circ dg, $$
and $dg$ is a onto,  $h$ is  J-holomorphic.

The zero set $Z$ of the holomorphic function $\det h \colon V \to \C$ is a proper complex subvariety of 
$V$ (of lower dimension); replacing $V$ by  $V \smallsetminus Z$ we get that  
the map $h^{-1}$ is also J-holomorphic; see Lemma \ref{invtheo}. 
The set $$U_A=h(g(U))$$ is dense in $A=f(M)$ since 
$$A=f(M)=f({\cl}_{\C}({U})) \subseteq \cl({f(U)}) =  \cl({h(g(U))}),$$
where ${\cl}_{\C}(\cdot)$ is the Zariski-closure and $\cl(\cdot)$ is closure in the $C^{\infty}$-topology.
$A \smallsetminus U_A$ is contained in the image $P_A$ of $f$ restricted to $P=M \smallsetminus U$. 
 \end{proof}

{\bf Notation:}  

We will call such $U_A$ an \emph{umbra} of the holomorphic shadow $A$, and 
call \emph{penumbra} an 
holomorphic shadow as in part
$(3)$ of the lemma.   The compact complex variety $N_A=N$ in which
 $h^{-1}[U_A]=V_A$ is dense is called a \emph{shadow caster} of $A$. 
 We call $g \colon U \rightarrow N$
 the \emph{map induced by the complex equivalence relation ${\sim}_f$}.

 \begin{Lemma} \label{invtheo}
    \emph{The inverse function theorem in the almost complex category.} 
              Suppose that $ f \colon X \rightarrow Y$ is a J-holomorphic map whose 
              derivative $df_p$ at the point $p$ is an isomorphism. 
              Then $f$ is a local J-holomorphic isomorphism at $p$.\\
       \end{Lemma}

\begin{proof}
Suppose that $ f \colon X \rightarrow Y$ is a J-holomorphic map whose derivative $df_p$ at the 
point $p$ is an isomorphism. 
By the inverse function theorem in the smooth category, locally there 
exists an inverse map $f^{-1}$ to $f$. 
It is enough to notice that 
  \[df_{p}\circ {J_{1}}_{p} = {J_{2}}_{f(p)}\circ df_{p}\]
implies, 
\[{J_{1}}_{p} \circ {df_{p}}^{-1}= {df_{p}}^{-1} \circ {J_{2}}_{f(p)},\]
i.e.,  $f^{-1}$ is J-holomorphic.
\end{proof}

 \begin{Remark} \label{resres} 
Let $M,N$ be compact complex manifolds, $U \subset M$, $V \subset N$ open dense subsets, 
and $g \colon U \to V$ a holomorphic map whose graph $G \subset M \times N$ 
is constructible, 
i.e., it is a Boolean combination 
of complex analytic subvarieties of $M \times N$. 
Resolve $\Gamma$ to a compact complex manifold $\tilde{\Gamma}$ by Hironaka's \cite{hiro} 
resolution of singularities $\phi \colon \tilde{\Gamma} \to \Gamma$.  
The set $U$ is naturally embedded in $\Gamma \subset M \times N$. 
The proper transform $\phi^{-1}U$ of $U$ is a Zariski-open dense subset of $\tilde{\Gamma}$. 
By restricting $U$ we may assume that $\phi|_{\phi^{-1}U}$ is an isomorphism onto $U$. 
Composing $\phi|_{\phi^{-1}U}$ with the projection $\pi_N \colon M \times N \to N$ we get 
the map $g$. Thus the holomorphic map 
$\tilde{g}=\pi_N \circ \phi$ can be considered an expansion of $g$.

Now, consider a holomorphic shadow $A \subset X^n$ (in $\H$), that is an image of a $J$-holomorphic 
$f \colon M \to X^n$ (in $\F$). 
The map $g \colon U \to N$ induced by the complex equivalence relation 
${\sim}_f$ is constructible. 
(See \cite[Section 2.2]{moosa}.)
Hence, (by expanding $g$ to $\tilde{g}$ and replacing $M$ by $\tilde{\Gamma}$), 
we may assume that the map induced by ${\sim}_f$ is $g \colon M \to N$.

 \end{Remark}

\begin{Remark} \label{psi}

Given a J-holomorphic 
$f \colon M \to X^{n+k}$, 
and a projection $\pi \colon X^{n+k} \to X^n$, let $U_S$ and $U_{\pi(S)}$ be umbras constructed as in Lemma 
\ref{greq}. 
We can assume that the restriction of $\pi$ 
to $U_S$ is a holomorphic and proper projection onto $U_{\pi(S)}$. 
To see this, first apply Lemma \ref{greq} and remark \ref{resres} to get  
$g_1 \colon M \to N_1$, and $g_2 \colon M \to N_2$, such that
for $i=1,2$, the map $g_i$ restricted to a Zariski-open dense subset $U$ of $M$
is a submersion onto 
a Zariski-open dense subset $V_i$ of $N_i$, and that ${g_i|_{U}}^{-1}(V_i)=U$.
Locally, (up to a holomorphic isomorphism), there are systems of holomorphic coordinates, 
in which $g_i \colon U \to V_i$ is given by the projection 
$(u_1,\ldots,u_l) \to (u_1,\ldots,u_{k_i})$; where $l > k_1 > k_2$. 
This gives a map $\psi \colon V_1 \to V_2$, 
mapping $(u_1,\ldots,u_{k_1})$ to $(u_1,\ldots,u_{k_2})$.
Applying Remark \ref{resres}, we expand $\psi$ to 
$\tilde{\psi} \colon \tilde{N_1} \to N_2$ between compact complex manifolds
(hence proper), using a resolution of singularities $\phi \colon{\tilde{N_1}} \to N_1$.
 By restricting $U,V_1,V_2$, we assume that $\phi|_{\phi^{-1}{V_1}}$ is an isomorphism onto $V_1$,
 and $\tilde{\psi}^{-1}{V_2}= \phi^{-1}{V_1}$, (we denote $\phi^{-1}{V_1}$ again by $V_1$).  
So 
the following diagram commutes.

 \xymatrix{
 X^{n+k}  &  M \ar[l]_{f} 
 & U  \ar[l]_{\supset} \ar[d]_{{\phi^{-1}}|_{V_1} \circ g_1} \ar[r]^{\Id}
  &  U \ar[d]^{g_2} \ar[r]^{\subset} & M  \ar[r]^{\pi \circ f} & X^n\\
  & \tilde{N_1}  & {V_1} \ar[ull]^{h_1} \ar[l]_{\supset} \ar[r]^{\tilde{\psi}} 
  &  V_2 \ar[urr]_{h_2} \ar[r]^{\subset} & N_2  }

The map $\psi$ as the restriction ${\tilde{\psi}}_{V_1} \colon V_1 \to V_2$ is proper.
(A compact set $A$ in the open set $V_2$ is compact in $N_2$, so 
$A'=\tilde{\psi}^{-1}A=\psi^{-1}A$ is 
 compact in $N_1$ and contained in $V_1$, hence is compact in $V_1$.)

\end{Remark}

\begin{Claim} \label{alun} Let $A$ be a holomorphic shadow in $\H$. If $A_1$ and $A_2$ are shadow casters (or shadow umbras) 
of $A$, then there is a 
holomorphic isomorphism  between  open dense subsets of them. \end{Claim}  

\begin{Claim} \label{irred} Let $C$ be a holomorphic shadow in $\H$.  
Then $C$ is irreducible as a holomorphic shadow if and only if a shadow caster of $C$ is irreducible as a 
complex variety. 
\end{Claim}

\begin{Proposition} \label{stra}
 A holomorphic shadow $S \subset X^n$ in $\H$ can be decomposed as 
 $$S=S^{(r)} \cup S^{(r-1)} \cup \ldots \cup S^{(0)},$$
 where for all $i$, $S^{(i)}$ is an i-dimensional integrable almost complex 
submanifold 
 of $X^n$, and $\cl({S^{(i)}}) \subseteq S^{(i)} \cup S^{(i-1)} \cup \ldots \cup S^{(0)}$.
 (Here $\cl$ can be interpreted  in the Z-topology or in the  $C^{\infty}$ topology.
 \end{Proposition}

\begin{Definition}\label{dim}

 If $S^{(r)} \neq \emptyset$, 
we say that $r$ is the \emph{dimension} of $S$, and denote it by $\dim S$. 
\end{Definition}

The dimension of a holomorphic shadow $C \in \H$, 
equals the dimension as a complex manifold of a 
shadow umbra $U_C$ of $C$, which equals the dimension of a related shadow caster $N_C$.  
By Claim \ref{alun}, the dimension of a holomorphic shadow is well defined.

The following claim is clear from our definition of dimension.

\begin{Claim} \label{basicdim}
Let $C_1$ and $C_{2}$ be holomorphic
shadows in $\H$, then
 \begin{description}
  \item [(DP)] The dimension of a point is $0$;
        \item [(DU)] $\dim(C_{1} \cup  C_{2}) = \max\{\dim C_{1},\dim C_{2}\};$
  \end{description}
\end{Claim} 

Claim \ref{irred} and axiom (DI) for subvarieties of a compact complex manifold imply the following claim.
\begin{Claim} \label{di}
(DI) If $C_{2}$ is irreducible and $C_{1} \subseteq C_2, C_{1} \neq C_2$,    
              then $\dim C_{1} < \dim C_{2}$;
\end{Claim}

\begin{Claim} \label{FC}
        (FC) Let 
             $S \in \H$ be a holomorphic shadow in $X^{n+m}$.
             Let $pr$ stand for the projection $X^{n+m} \rightarrow X^{n}.$
             Then, 
\begin{equation} \label{psk}
p(S,k)=\{a \in X^{n} \mid \dim(S \cap {pr}^{-1}(a)) > k \}
\end{equation}
is h-constructible. 
If $k \geq \min \dim({pr}^{-1}a  \cap S)$, then $p(S,k)$ is contained in a holomorphic shadow that is  a proper subset of $pr(S)$.
\end{Claim}
 
Notice that by Corollary \ref{inv}, there is a meaning to \eqref{psk}.

{\bf Notation:}
We say that a set is \emph{h-constructible} if it is constructible from holomorphic shadows, i.e., of the form $\cup_{i \leq k} A_i \smallsetminus B_i$, where $k$ is a natural number, $A_i$, $B_i$ are holomorphic shadows, and $B_i \subseteq A_i$.

\begin{proof}
First, we notice that for the decomposition $S = {\cup}_{i=1}^{n}{S_i}$ to irreducible components, 
$\dim({pr}^{-1}(a) \cap S)=\max \dim({pr}^{-1}(a) \cap S_i)$.
Hence    
$p(S,k) = {\cup}_{i=1}^{n}{p({S_i},{pr}|_{{S_i}},k)}$. So we may assume that $S$ is irreducible. 
Hence, by Claim \ref{prirred} and Claim \ref{irred},  so are $pr(S)$, and the shadow casters $N_S$ and $N_{pr(S)}$.

By Remark \ref{psi}, we may assume that the restriction of $pr$ to the umbra $U_S$ is a 
holomorphic and proper projection onto the umbra $U_{pr(S)}$. We identify $U_S$ and $U_{pr(S)}$ with the corresponding isomorphic Zariski-open sets in the shadow casters $N_S$ and $N_{pr(S)}$. For every $a \in pr(U_S)(=U_{pr(S)})$, we have that ${pr}^{-1}(a) \cap U_S$ is an umbra in $p^{-1}(a) \cap S$ hence $\dim ({pr}^{-1}(a) \cap S) = \dim  ({pr}^{-1}(a) \cap U_S)$.
Assume that $pr|_{U_S}$ is given by holomorphic functions
$p_1,...,p_l$. For a point in the fiber $U_S \cap pr^{-1}a$, the fact that  $\dim( pr^{-1}a \cap U_S) >k$ 
implies that 

(*) the complex rank of the Jacobian of $(p_1,...,p_l)$ at $b$ is $\leq \dim{U_S} - k$. 

This condition is equivalent to vanishing of all $\dim{U_S}-k$
minors of the Jacobian. Let $Z$ be defined by (*). Let $W$ be the Zariski closure in $N_S$ of $Z$.
Let $S'$ be a penumbra in $S$ such that $S \smallsetminus U_S \subset S'$, and $(pr(S))'$ be a penumbra in $pr(S)$ such that $pr(S) \smallsetminus U_{pr(S)} \subset (pr(S))'$. Then $p(S,k)=p(S',k) \cup ((pr(S) \smallsetminus (pr(S))') \cap f(W))$, (where $f \colon M_S \to X$ is the map giving the holomorphic shadow $S=f(M_S)$). By induction $p(S',k)$ is h-constructible, hence so is $p(S,k)$.

\end{proof}

\begin{Claim} \label{ADF}
 (ADF) Let 
$S \in \H$ be an irreducible  
  holomorphic shadow in $X^{n+m}$.
 Let ${pr}$ denote the canonical projection $X^{n+m} \rightarrow X^{n}.$ 
 Then  
\begin{equation} \label{adfeq}
\dim S = \dim pr(S) + \min_{a\in pr(S)}\{\dim ({pr}^{-1}(a) \cap S)\}.
\end{equation}
 \end{Claim}

\begin{proof}
By Remark \ref{psi}, we may assume that the restriction of $pr$ to $U_S$ is a 
proper holomorphic projection onto $U_{pr(S)}$.
By the corresponding claim for complex analytic subvarieties in a complex manifold, 
this is true for shadow umbras $U_S$ and $U_{pr(S)}$. 
So, $$\dim U_S= \dim U_{pr(S)} + \min_{a\in U_{pr(S)}}\{\dim ({pr}^{-1}(a) \cap U_S)\}.$$
For every $a \in pr(U_S)(=U_{pr(S)})$, we have that ${pr}^{-1}(a) \cap U_S$ is an umbra in $p^{-1}(a) \cap S$ hence $\dim ({pr}^{-1}(a) \cap S) = \dim  ({pr}^{-1}(a) \cap U_S)$.
By Claim \ref{FC}, the minimal dimension of a fiber ${pr}^{-1}(a) \cap S$ cannot be attained on a holomorphic shadow that is a proper subset of the irreducible shadow $pr(S)$. Therefore,
 $$\min_{a\in U_{pr(S)}}\{\dim ({pr}^{-1}(a) \cap U_S)\}=
\min_{a\in pr(S)}\{\dim ({pr}^{-1}(a) \cap S)\}.$$
Since $\dim S = \dim U_S$, and $\dim pr(S) = \dim U_{pr(S)}$, 
we get \eqref{adfeq}. 

\end{proof}

\begin{Remark}
It follows from the axioms that for any holomorphic shadow $S \subseteq X^{n+m}$,
$$\dim S \geq \dim pr(S) + \min_{a\in pr(S)}\{\dim ({pr}^{-1}(a) \cap S)\}.$$ See \cite{imam1} - Fact 2.2. 
\end{Remark}

{\bf Notation:}
If $X$ is a holomorphic shadow, its dimension is already defined. Otherwise,
we assign $\dim X$ to be half the dimension of $X$ as a real manifold.
  The dimension of $ {{\Delta}^{n}}_{(i_{1},\ldots,i_{k})}$ is $(n-k +1) {\dim X}$.

 For a shadow $S \in \H$ in $X^l$ and a diagonal $D$ in $X^{n-l}$, 
the \emph{dimension} of $S \times D$ is the sum of $\dim S$ (as a holomorphic shadow)
and $\dim D$ (as above);
the \emph{dimension} of the image of $S \times D$ under permutations of the coordinates is 
the dimension of $S \times D$. The dimension of a finite union of such sets is the maximum of 
the dimensions of the sets in the union. 

It is easy to check that the dimension axioms still hold in ${\s}_{(X,J,\H)}$.  This completes the proof of Theorem \ref{genz}.

\begin{Remark}
By the arguments in the proof of Theorem \ref{genz}, we get that  ${\s}_{(X,J,\H)}$ with dimension as in Definition \ref{dim}  is a Zariski-type structure, with the (EU) property, whenever $\H$ consists of the images of a collection of J-holomorphic maps $\F$ such that $\F$ satisfies the pulling back diagonals property, $\F$ is closed under compositions and inverses, and $\F$ contains all the coordinate projections.  
\end{Remark}

\subsection*{Almost complex manifolds that are 
ample as Z-structures are of real dimension two}

We say that a Zariski-type structure on $X$ is \emph{very ample} if
there exists a family $Y \subset X^2 \times X^n$ of  irreducible one-dimensional Z-closed subsets of
 $X^2$, parametrized by a Z-closed irreducible set in $X^n$, 
such that 
\begin{itemize}
\item through any  two points in $X^2$, there is a curve in the family passing through both, and 
\item for any  two points in $X^2$, there is a curve in the family passing through exactly one of the points.
\end{itemize} 
 If only the first condition is satisfied, the structure is called \emph{ample}. 

\begin{Claim} \label{genamp}

Let $X$ be a  manifold of real dimension $2n$, and $J \in {\J(X)}_{\gen}$. If $(X,{\s}_{(X,J)})$ is ample, then $\dim_{\R} X \leq 2$, in particular $J$ is integrable.
\end{Claim}

\begin{proof}
Assume $\dim_{\R}X >2$.
Since $J \in {\J(X)}_{\gen}$, the projection of a
holomorphic shadow $Y \subset X^k$ on every coordinate is either a point or a
$J$-holomorphic curve, hence, since $\dim_{\R}X >2$, a holomorphic shadow $Y \subset X^2 \times X^n$ does not project onto $X^2$. Similarly, if $Y$ is a diagonal, a Cartesian product of holomorphic shadows and diagonals, or the image of such a set under coordinate-permutation, it does not project onto $X^2$. Thus there is no $Y \subset X^2 \times X^n$ that can serve as a family demonstrating ampleness.
\end{proof}

To generalize this claim we use results on Zariski geometries.

\emph{Zariski geometry} is defined by \cite{zilber}, \cite{ezilber}.
A set $X$ with a collection of compatible Noetherian topologies, one on each $X^n$, $n \in \N$, and the Noetherian dimension as dimension, such that $X$ is irreducible and of Noetherian dimension one, and 
the Pre-smoothness (PS) property is satisfied, is a (one-dimensional) Zariski geometry. 
(A topological space has 
           \emph{Noetherian dimension} 
            $n$ if $n$ is the maximal length of a chain of closed irreducible sets
             $C_n \supset C_{n-1} \supset \ldots \supset C_0$.)

Any smooth algebraic curve $X=C$ can be viewed as a Zariski geometry; 
the Z-closed subsets are taken to 
be the Zariski closed subsets of $C^n$ for each $n$; this Zariski geometry is very ample.

In \cite[Theorem 1]{ezilber}, Hrushovski and Zilber show that if $X$ is a very ample Zariski geometry, 
then there exists a smooth curve $C$ over an algebraically 
closed field $F$, such that $X$ and $C$ are isomorphic as Zariski geometries. 
 In \cite[Theorem 2]{ezilber}, they show that if $X$ is an ample Zariski geometry, then there exists an algebraically 
closed field $F$, and a surjective map $f \colon X \to \oCP(F)$, such that off a finite set $f$ induces a closed continuous maps on each Cartesian power.

\begin{noTitle}
Let $X$ be a  $2n$-manifold, and $J$ an almost complex structure on $X$. Assume that the pulling back diagonals property is satisfied for $J$-holomorphic maps from compact complex manifolds to Cartesian products of $X$.  Thus the Z-axioms (L1)--(L5), (P), and (DCC), and the (EU) property, are satisfied in the holomorphic shadows structures. Assume also that $X$ is of Noetherian dimension one in this structure, i.e., 
any proper Z-closed subset of $X$ is a finite set of points.

If $X$ itself is a holomorphic shadow, then 
outside the penumbra, a proper Z-closed subset hence a finite set, it is a complex manifold (the umbra). 

If $X$ is not a holomorphic shadow, then it does not contain holomorphic shadows except for finite sets of points. In this case, 
for any $n \in \N$, the space $X^n$ contains no holomorphic shadows (that are not finite) either, since 
for an infinite holomorphic shadow in $X^n$ its image under one of the coordinate projections
$X^n \to X$ is an infinite holomorphic shadow in $X$;
  thus the holomorphic shadows structure is trivial, i.e., consists only of diagonals and points.  

In both cases we have the Pre-smoothness (PS) property, with Noetherian dimension. 
\end{noTitle}

\begin{Corollary} \label{zarone}
Let $X$ be a    manifold, and $J$ an almost complex structure on $X$. Assume that the pulling back diagonals property is satisfied for $J$, and that $X$ is of irreducible and of Noetherian dimension one in the holomorphic shadows structure.
Then this structure, with Noetherian dimension, is a Zariski geometry that also satisfies the (EU) property.
\end{Corollary}

\begin{Corollary} \label{compl}
Let $X$ be a    manifold, and $J$ an almost complex structure on $X$. Assume that the pulling back diagonals property is satisfied for $J$, and that $X$ is of irreducible and of Noetherian dimension one in the holomorphic shadows structure. If this structure is ample, then
$(X,J)$ is of real dimension two, in particular $J$ is integrable.
\end{Corollary}

\begin{proof}

By Theorem 2 in \cite{ezilber}, there exists an algebraically closed field $\F$ and a map 
${\pi} \colon X \rightarrow {P}^{1}(\F)$. The map $\pi$ maps constructible sets to
algebraically constructible sets; 
off a certain finite set, $\pi$ is surjective and induces a closed 
continuous map on each Cartesian power.
$\F$ is interpretable on $X$, i.e., there is an equivalence relation $\sim _{\pi}$ 
on $X$, such that for some finite subset $A'$, the quotient by $\sim _{\pi}$   
of $\bar {X} \times \bar {X}$, where
 $\bar {X}=X - A'$,  is a closed subset of $\bar {X} \times \bar {X}$. 
There are definable subsets $A, M \subset \bar {X} \times \bar {X} \times \bar {X}$ such
that their quotient 
by $\sim _{\pi}$, restricted to products of coordinate neighbourhood,
give the graphs of the field operations (addition and multiplication) in $\F$.


By removing finite sets from $\bar{X}$ and $\F \subseteq P^{1}(\F)$, we have 
$\pi \colon \tilde{X}\rightarrow \tilde{\F}$ that is surjective, continuous, 
maps constructible sets to algebraically constructible sets, and finite to one.

``Finite to one'' follows from the fact that 
the Zariski geometries $X$ and $P^1(\F)$ are both of Noetherian dimension 1,  
and in a generic point 
$y$ in $P^1(\F)$, 
$\dim_{X}(X)= \dim_{P^{1}(\F)}({\pi}(X))+\dim_X({\pi}^{-1}y)$.

For $k \in \N$, let $E_k=\{f \in \tilde{\F} |{\pi}^{-1}(f) \cap \tilde{X}| = k \}$.
$E_k$ is a definable set in $P^{1}(\F)$ as interpreted in $X$.
Since $P^{1}(\F)$ is a strongly minimal set, either $E_k$ is finite 
or $\F-E_k$ is finite.
 If for every $k \in \N$,
$E_k$ is finite we get that $\F$ is countable, (recall that $\pi$ is finite to one),
 contradicting  axiom (EU).
($\F$ is infinite since it is algebraically closed.)
Thus there exists $n \in \N$ such that for $f \in \tilde{\F}- \text{ a finite set }$
(to be denoted also $\tilde{\F}$),  
$|({\pi}^{-1}(f) \cap \tilde{\F}) |= n$.

For $p \in \tilde{\F}$, take a coordinate neighbourhood $U \subset \tilde{X}$ 
around of a point in the fiber 
$\pi^{-1}p$, and define its image in $\tilde{F}$ by $\pi$ to be a coordinate neighbourhood. 
This gives $\tilde{\F}$ a manifold structure.   
The map $\pi$ is continuous 
with respect to this topology on $\tilde{\F}$ and the given topology on $\tilde{X}$ (induced from 
the topology on $X$). 

Moreover, covering $\F$ by translates of 
$\tilde{\F}$, (by the translations $f \to f+b$, $f \to f \ast b$, 
induced from the addition and multiplication in $\F$),
we obtain $\F$ as a manifold. 
The field operations $+ \colon \F \times \F \to \F$ and $\ast \colon \F \times \F \to \F$ 
are continuous with respect to this topology, as the graph of the operations restricted to 
products of coordinate neighbourhoods are given by the quotient 
by $\sim _{\pi}$ of  constructible subsets 
$A, M \subset \tilde {X} \times \tilde{X} \times \tilde {X}$ , as above. 
(By elimination of quantifiers, a set is constructible iff it is definable.)
So the manifold structure on $\F$ is consistent with the field operations.

In particular $\F$ is locally compact. Since it is also algebraically closed, $\F=\C$.
Since $\pi \colon \tilde{X} \rightarrow \tilde{F}$ is finite to one 
and ${\pi}^{-1}{(\F -\text{ finite set})}$ is open in $X$, 
the almost complex manifold $X$ must be of real dimension $2$.

 \end{proof}

\section{Holomorphic shadows in symplectic geometry} \label{sympl}

The theory of $J$-holomorphic curves has been an active study area and a powerful tool in symplectic geometry, 
since the pioneering paper of Gromov \cite{gromovcurves}.

A \emph{symplectic} structure on a smooth $2n$-dimensional manifold $X$ is a closed $2$-form $\omega$ which is 
non-degenerate (i.e., ${\omega}^n$ does not vanish anywhere). Two symplectic manifolds $(X_1,\omega_1)$ and $(X_2,\omega_2)$ 
are called \emph{symplectomorphic} if there is a diffeomorphism $\phi \colon X_1 \to X_2$ such that $\phi^{*}{\omega_2}=\omega_1$.  
A symplectic form $\omega$ is said to \emph{tame} an almost complex structure $J$ 
if $\omega$ is $J$-positive, i.e., $$\omega(v,Jv) > 0$$ for all non-zero $v \in TX$.
This implies that for every embedded submanifold $C \subset M$,
if $J(TC) = TC$ then $\omega|_{TC}$ is non-degenerate.
 Given $\omega$, we denote (in this section) by $$\J=\J(X,\omega)$$ the space of almost complex structures $J$ on $X$ that are tamed by $\omega$. 
The space $\J$ is nonempty and contractible \cite[Proposition~2.50(iii)]{MS:intro}; in particular, it is path connected.  As a result, the first Chern class of the complex vector bundle $(TX,J)$ 
is independent of the choice of $J \in \J$.

We say that $A \in H_2(M;\Z)$ is \emph{$J$-indecomposable} if it does not split as a sum 
$A_1 + \ldots + A_k$ of classes all of which can be represented by non-constant $J$-holomorphic curves. 
The class $A$ is called \emph{indecomposable} if it is 
$J$-indecomposable for all $\omega$-tame $J$. Notice that if $A$ cannot be written 
as a sum $A = A_1 + A_2$ where $A_i \in H_2(M;\Z)$
and $\int_{A_i}{\omega} > 0$,   then it is indecomposable.

\subsection*{Restating claims in the language of Shadows structures}

We now restate claims of Gromov \cite{gromovcurves} and McDuff 
\cite{McDuff-Structure} in the language of Zariski structures; 
we will sketch some of the ideas of the (geometric) proofs.

{\bf Notation:}
An \emph{isomorphism} of two Zariski structures $\z_{1}(X)$, $\z_{2}(Y)$ is a map $\z_{1}(X) \to \z_{2}(Y)$ which is an isomorphism 
of topologies between $X^n$ and $Y^n$ for all $n \in \N$, that commutes with coordinate projections, Cartesian products and dimension assigning.
An \emph{embedding} of one Zariski structure into the other is a one-to-one map that is an isomorphism with its image.

We will say that a Zariski structure $\z$ is embedded into a structure (not necessarily a Zariski structure) $\s$ with a (partial) 
dimension function if there is a one-to-one map from the closed sets in $\z$ to the closed sets in $\s$, that preserves the subset relation between closed sets as well as projections, such that the restriction of $\s$ to its image is 
a Zariski structure, with which the map is an isomorphism.

\begin{Example}
Consider $(S^2 \times S^2, J_0 \oplus J_0)$, where $J_0$ is the standard complex structure on the sphere $S^2=\oCP$.
Denote by $\omega_0$ an area form on the sphere $S^2$, whose orientation agrees with the orientation induced 
by $J_0$. 
The form $\omega_0 \oplus \omega_0$, defined as the sum of the pullbacks of $\omega_0$ to $S^2 \times S^2$ via the coordinate projections, 
is a symplectic form on $S^2 \times S^2$ that tames $J_0 \oplus J_0$.

Denote by $\Striv (S^2 \times S^2)$ the structure generated by finite unions and Cartesian products from 
\begin{enumerate}
\item \label{1a} points $(s,r) \in S^2 \times S^2$,
\item \label{1b} the set $S^2 \times S^2$, 
\item \label{1c} the sets $\{s\} \times S^2$, $S^2 \times \{s\}$, for any $s \in S^2$,
\item \label{1d} the diagonals $${\Delta}^{n}_{i_1, \ldots, i_k}=\{(x_1,\ldots,x_n) \vert x_{i_1}=\ldots=x_{i_k}, x_i \in S^2 \times S^2\};$$
\end{enumerate}  
with the natural dimension assigned: 
\begin{itemize}
\item the dimension of a set in (\ref{1a}) is $0$, 
 \item the dimension of a set in (\ref{1b}) is $2$, 
 \item the dimension of a set in (\ref{1c}) is  $1$, 
\item $\dim \Delta^{n}_{i_1,\ldots,i_k}=2(n-k+1)$,
\item $\dim(S_1 \cup S_2) = \max \{\dim S_2 , \dim S_2\}$,
\item $\dim S \times T = \dim S + \dim T$.
\end{itemize}
Then $\Striv (S^2 \times S^2)$ is a Zariski structure. 

For any $s \in S^2$, the sphere $\{s\} \times S^2$ is embedded as a symplectic sphere in 
$(S^2 \times S^2, \omega_0 \oplus \omega_0)$, and as a $J_0 \oplus J_0$-sphere. This implies that
\begin{Claim} 
 $\Striv (S^2 \times S^2)$ is embedded as a Zariski structure in the holomorphic shadows structure $\s_{(S^2 \times S^2, J_0 \oplus J_0)}$. 
\end{Claim}

We show a similar claim for $J \in \J(S^2 \times S^2, \omega_0 \oplus \omega_0)$ that does
not necessarily split as a product of almost complex structures on $S^2$. 
This follows from known results of Gromov \cite[2.4 A1]{gromovcurves} and \cite[Lemma 4.1, Lemma 4.2]{McDuff-Structure}. I prove Claim \ref{wdf} here to give the reader the flavor of the arguments.

\begin{Claim} \label{wdf}
For every almost complex structure $J$ that is tamed by $\omega_0 \oplus \omega_0$, for each point in  $S^2 \times S^2$ there exists and unique an embedded $J$-holomorphic sphere in 
$$A=[S^2 \times \pt]$$
$$ (B=[\pt \times S^2]).$$
In addition, any $J$-holomorphic sphere in $A$ and any $J$-holomorphic sphere in $B$ intersect once and transversally.

\end{Claim}

\begin{proof}

Through each point in $S^2 \times S^2$ there is an embedded $J_0$-holomorphic sphere 
$f_0 \colon S^2 \to S^2 \times S^2$ in $A=[S^2 \times \pt]$
$ (B=[\pt \times S^2]).$ In particular, $A$ is the homology class of an embedded $\omega_0 \oplus \omega_0$-sphere of minimal symplectic area, hence 
$A$ is indecomposable,
and every $J$-holomorphic sphere $f \colon S^2 \to S^2 \times S^2$ in $A$ is simple.

By the adjunction inequality (in a four-dimensional manifold), 
if $A$ is represented by a simple 
J-holomorphic curve $f$, then
$$ A \cdot A - c_1(A) + 2 \geq 0,$$
with equality if and only if $f$ is an embedding;
see~\cite[Cor.~E.1.7]{MS:JCurves}.
Applying this to $(f_0,J_0)$, we get that the homology class $A$
satisfies $A \cdot A - c_1(A) + 2 = 0$.
Applying the adjunction inequality to any $(f,J) \in \M(A,S^2,\J)$, 
we get that $f$ is an embedding, 
hence, by the Hofer-Lizan-Sikorav regularity criterion, $(f,J)$ is regular for $p_{A}$. 
(The Hofer-Lizan-Sikorav regularity criterion asserts that if $f$ is an 
immersed $J$-holomorphic curve in a four-dimensional manifold 
and $c_{1}([f])\geq 1$, then $(f,J)$ is a regular point 
for the projection $p_{A}$ \cite{HLS}.)
By the implicit function theorem, every $p_{A}$-regular sphere $(f,J)$ persists when $J$ is perturbed, see e.g., \cite[Remark 3.2.8]{MS:JCurves}.
On the other hand, since
 $A$ is indecomposable, 
Gromov's compactness theorem \cite[1.5.B.]{gromovcurves} implies that  
if $J_n$ converges in $\J$, then every sequence $(f_n,J_n)$ in
$\M(A,S^2,\J)$
has a $(C^{\infty}-)$convergent 
subsequence.
  We conclude that for each point $\pt \in S^2 \times S^2$, the set of $J \in \J(=\J(S^2 \times S^2,\omega_0))$ for which there is an embedded 
$J$-holomorphic sphere through $\pt$ in $A=[S^2 \times \pt]$ ($B=[\pt \times S^2]$) is nonempty open and closed in the connected space $\J$, hence it equals $\J(=\J(S^2 \times S^2,\omega_0 \oplus \omega_0))$.

Let $J \in \J$. If there are two different simple $J$-holomorphic spheres $f_1,f_2$ through a point in $S^2 \times S^2$, then, by positivity of intersections in almost complex four-manifolds \cite[ Theorem E.1.5]{MS:JCurves},  the intersection number $[{f_1}(S^2)] \cdot [{f_2}(S^2)]$ is positive. Thus, since, $A \cdot A=0=B \cdot B$, there cannot be two different $J$-holomorphic spheres in $A$ ($B$) through a point in $S^2 \times S^2$. 
Also, again by positivity of intersections, for simple $J$-holomorphic spheres $f_1,f_2$, the intersection number $[{f_1}(S^2)] \cdot [{f_2}(S^2)]$ equals $1$ if and only if the spheres meet exactly once and transversally.  Hence a $J$-holomorphic sphere in $A$ and a $J$-holomorphic sphere in $B$ intersect once and transversally.

\end{proof}

In the language of shadows structures, this claim have the following translation.

\begin{Claim} \label{clemb}
$\Striv (S^2 \times S^2)$ can be embedded into the shadows structure $\s_{(S^2 \times S^2,J)}$ for every
$J \in \J(S^2 \times S^2,\omega_0 \oplus \omega_0)$. 

\end{Claim}

\begin{proof}
Fix $J \in \J(S^2 \times S^2, \omega_0 \oplus \omega_0)$. Choose $s_0 \in S^2$.   By Claim \ref{wdf}, there is a unique $J$-holomorphic curve $C_1$ in $A$ ($C_2$ in $B$) through $(s_0,s_0)$. Choose a $J|_{C_1}$-holomorphic diffeomorphism $a_1$ of $C_1$ onto $S^2$ such that $a_1(s_0,s_0)=s_0$, and a $J|_{C_2}$-holomorphic diffeomorphism $a_2$ of $C_2$ onto $S^2$ such that $a_2(s_0,s_0)=s_0$.  
Now, send $S^2 \times \{s\}$ to the unique curve in $A$ that intersects $C_2$ in $v$ such that $a_2(v)=s$, and send $\{s\} \times S^2$ to the unique curve in $B$ that intersects $C_1$ in $v$ such that $a_1(v)=s$. This is a well defined and one-to-one map; it maps the two  families $\{s\} \times S^2$, $S^2 \times \{s\}$, $s \in S^2$ to two families
of
 $J$-holomorphic spheres such that each member of one family intersects each member of the other exactly once and transversally. 
This induces an embedding of $\Striv (S^2 \times S^2)$ into the shadows structure $\s_{(S^2 \times S^2,J)}$; (a point $(r,t)$ is sent to the intersection point of the $J$-holomorphic sphere in $A$ that is the image of $S^2 \times \{t\}$ and the $J$-holomorphic sphere in $B$ that is the image of $\{r\}  \times S^2$, $S^2 \times S^2$ is sent to $S^2 \times S^2$, each diagonal is sent to itself).

\end{proof}

\begin{Remark}
The above embedding is not onto: for almost complex structures in $\J(S^2 \times S^2,\omega_0 \oplus \omega_0)$, there are $J$-holomorphic curves in $S^2 \times S^2$ in homology classes other than $[S^2 \times \{\pt\}]$ or $[\{\pt\} \times S^2]$, e.g., in $[\{s,s\}_{s \in S^2}]$.

\end{Remark}

Similar arguments  to the ones in Claim \ref{clemb} give the following claim. See \cite[Lemma 4.2, Lemma 4.6]{McDuff-Structure}.
\begin{Claim}
Let $\omega$ be a symplectic form on $S^2 \times S^2$ such that there exist
symplectically embedded spheres in $A=[S^2 \times \pt]$ and $B=[\pt \times S^2]$ that intersect exactly once and transversally.      
Then $\Striv (S^2 \times S^2)$ can be embedded into the shadows structure $\s_{(S^2 \times S^2,J)}$ for every
$J \in \J(S^2 \times S^2,\omega)$. 
 \end{Claim}

\end{Example}
\eoe

\begin{Example}
We denote by $\omega_{\FS}$ the Fubini-study form on $\tCP$.

Denote by $\Striv (\tCP)$ the structure generated by finite unions and Cartesian products from 
\begin{enumerate}
\item \label{2a} the points of $\tCP$, with dimension assigned $0$,
\item \label{2b} the set $\tCP$, with dimension assigned $2$,
\item \label{2c} a family $\F=\{C(p_0,p)\}_{p \in \oCP}$, where $p_0$ is a fixed point in $\tCP$ and $p$ is a point on a $\oCP$-line $L$ in $\tCP$ such that $p_0$ is not on $L$,  of (spheres) $C(p_0,p)$ in $\tCP$ such that for $p \neq q$, the intersection
$C(p_0,p) \cap C(p_0,q)$ is the point $p_0 \in \tCP$; each $C(p_0,p)$ is assigned dimension $1$, 
\item  \label{2d} diagonals $\Delta^{n}_{i_1, \ldots, i_k}$ in $(\tCP)^n$, with dimension assigned $2(n-k+1)$.
\end{enumerate}  
Then $\Striv (\tCP)$ is a Zariski structure, embedded in the holomorphic shadows structure $\s_{(\tCP, J_0 \oplus J_0)}$, where $J_0$ is the standard complex structure on $\tCP$. 

\begin{Claim}
Let an almost complex structure $J$ on $\tCP$ be tamed by the standard symplectic form $\omega_{\FS}$ on $\tCP$. Then there is a $J$-holomorphic sphere $C \subset \tCP$ through two given points $v$ and $v'$ in $\tCP$ which  is homologous to the projective line $\CP^{1} \subset \tCP$. Since the algebraic intersection number between two such spheres equals one, any two of them, say $C$ and $C'$ in $\tCP$, necessarily meet at a single point, say at $v \in C \cap C'$, unless $C=C'$. Furthermore, the spheres $C$ and $C'$ are regular at $v$ and meet transversally. Hence, $C$ is regular at all points $v \in C$ and is uniquely determined by $v$ and $v' \neq v$. Moreover, $C=C(v,v')$ smoothly depends on $(v,v')$.
\end{Claim}
This is \cite[2.4 A]{gromovcurves}.
In the language of shadows structures this claim has the following translation.
\begin{Claim}
For every  $J \in \J(\tCP, \omega_{\FS})$, the structure $\Striv (\tCP)$ is embedded into $\s_{(\tCP,J)}$.
\end{Claim}

\end{Example}
\eoe

\newpage

\end{document}